\documentclass[11pt]{article}
\usepackage{authblk}
\usepackage{amssymb}
\usepackage[all]{xy}
\usepackage{mathrsfs}
\usepackage{amsmath,amssymb,amsthm}
\usepackage{extarrows}
\newtheorem{theorem}{Theorem}[section]
\newtheorem{lemma}[theorem]{Lemma}
\newtheorem{proposition}[theorem]{Proposition}
\newtheorem{definition}[theorem]{Definition}

\newtheorem{remark}[theorem]{Remark}

\newtheorem{example}[theorem]{Example}

\DeclareMathAlphabet{\mathpzc}{OT1}{pzc}{m}{it}
\def\to{\rightarrow}

\def\f{\mathfrak}

\def\c{\mathcal}

\def\r{\mathrm}

\def\ov{\overline}
\date{}
\begin{document}
%%%%%%%%%%%%%%%%%%%%%%%%%%%%%%%%%%%%%%%%%%%%%%%%%%%%%%%%%%%%%%%%%%%%%%%%%%%%%%%%%%%%%%%%%%%%%%%%%%%%%%%%%%%%%%%%%%%%%%%%%%
\title{Dynamical entropy of probability measures on\\ infinite product spaces}
%%%%%%%%%%%%%%%%%%%%%%%%%%%%%%%%%%%%%%%%%%%%%%%%%%%%%%%%%%%%%%%%%%%%%%%%%%%%%%%%%%%%%%%%%%%%%%%%%%%%%%%%%%%%%%%%%%%%%%%%%%
\author{Maysam Maysami Sadr\thanks{sadr@iasbs.ac.ir, corresponding author, orcid.org/0000-0003-0747-4180},
Mina Shahrestani\\\vspace{2mm}\&\\ Danial Bouzarjomehri Amnieh}
\affil{Department of Mathematics\\ Institute for Advanced Studies in Basic Sciences (IASBS)\\ Zanjan, Iran}
%%%%%%%%%%%%%%%%%%%%%%%%%%%%%%%%%%%%%%%%%%%%%%%%%%%%%%%%%%%%%%%%%%%%%%%%%%%%%%%%%%%%%%%%%%%%%%%%%%%%%%%%%%%%%%%%%%%%%%%%%%
%%%%%%%%%%%%%%%%%%%%%%%%%%%%%%%%%%%%%%%%%%%%%%%%%%%%%%%%%%%%%%%%%%%%%%%%%%%%%%%%%%%%%%%%%%%%%%%%%%%%%%%%%%%%%%%%%%%%%%%%%%
\maketitle
%%%%%%%%%%%%%%%%%%%%%%%%%%%%%%%%%%%%%%%%%%%%%%%%%%%%%%%%%%%%%%%%%%%%%%%%%%%%%%%%%%%%%%%%%%%%%%%%%%%%%%%%%%%%%%%%%%%%%%%%%%
\begin{abstract}
The aim of this note is to introduce a notion of dynamical entropy, which we call infinite-product entropy, for probability
measures on (countable) infinite cartesian product of any measurable space with itself. The idea behind the definition
is that any infinite product space may be considered as a type of dynamical object. We have considered in a previous note a similar
idea in topological dynamics to define a notion of dynamical entropy for arbitrary subsets of infinite products of compact topological spaces.
We consider some basic properties of infinite-product entropy, e.g. shift invariance, convexity, subadditivity with respect to product of
probability measures, the behavior with respect to dilation and restriction. We show that for a translation invariant probability measure
the infinite-product entropy coincides with the usual entropy of a shift transformation. We consider some basic examples and computations.
We also consider a variational inequality related to infinite-product entropy and topological entropy of subsets of infinite product spaces.
%28D20 Entropy and other invariants  28Dxx Measure-theoretic ergodic theory
%37B40 Topological entropy 37Bxx Topological dynamics
%37A50 Dynamical systems and their relations with probability theory and stochastic processes 37Axx Ergodic theory
%60-xx Probability theory and stochastic processes
\end{abstract}
%%%%%%%%%%%%%%%%%%%%%%%%%%%%%%%%%%%%%%%%%%%%%%%%%%%%%%%%%%%%%%%%%%%%%%%%%%%%%%%%%%%%%%%%%%%%%%%%%%%%%%%%%%%%%%%%%%%%%%%%%%
%%%%%%%%%%%%%%%%%%%%%%%%%%%%%%%%%%%%%%%%%%%%%%%%%%%%%%%%%%%%%%%%%%%%%%%%%%%%%%%%%%%%%%%%%%%%%%%%%%%%%%%%%%%%%%%%%%%%%%%%%%
%%%%%%%%%%%%%%%%%%%%%%%%%%%%%%%%%%%%%%%%%%%%%%%%%%%%%%%%%%%%%%%%%%%%%%%%%%%%%%%%%%%%%%%%%%%%%%%%%%%%%%%%%%%%%%%%%%%%%%%%%%
%%%%%%%%%%%%%%%%%%%%%%%%%%%%%%%%%%%%%%%%%%%%%%%%%%%%%%%%%%%%%%%%%%%%%%%%%%%%%%%%%%%%%%%%%%%%%%%%%%%%%%%%%%%%%%%%%%%%%%%%%%
%%%%%%%%%%%%%%%%%%%%%%%%%%%%%%%%%%%%%%%%%%%%%%%%%%%%%%%%%%%%%%%%%%%%%%%%%%%%%%%%%%%%%%%%%%%%%%%%%%%%%%%%%%%%%%%%%%%%%%%%%%
%%%%%%%%%%%%%%%%%%%%%%%%%%%%%%%%%%%%%%%%%%%%%%%%%%%%%%%%%%%%%%%%%%%%%%%%%%%%%%%%%%%%%%%%%%%%%%%%%%%%%%%%%%%%%%%%%%%%%%%%%%
%%%%%%%%%%%%%%%%%%%%%%%%%%%%%%%%%%%%%%%%%%%%%%%%%%%%%%%%%%%%%%%%%%%%%%%%%%%%%%%%%%%%%%%%%%%%%%%%%%%%%%%%%%%%%%%%%%%%%%%%%%
%%%%%%%%%%%%%%%%%%%%%%%%%%%%%%%%%%%%%%%%%%%%%%%%%%%%%%%%%%%%%%%%%%%%%%%%%%%%%%%%%%%%%%%%%%%%%%%%%%%%%%%%%%%%%%%%%%%%%%%%%%
\section{The main concept}\label{2110020632}
In this short note we introduce a notion of dynamical entropy $H^{ip}(\mu)$, which we call infinite-product entropy, for probability
measures $\mu$ on (countable) infinite cartesian product $(X^\infty,\Sigma^\infty)$ of an arbitrary measurable space $(X,\Sigma)$.
The idea behind the definition of $H^{ip}$ is that any infinite product space may be considered as a type of dynamical object.
Recently, we have considered in our previous note \cite{Sadr1} a similar idea in topological dynamics to define a notion of dynamical
entropy for arbitrary subsets of $X^\infty$ where $X$ is any compact topological space.

In this section we define $H^{ip}$ and prove that the ordinary measure-theoretic entropy of any measure-preserving transformation
on a probability space may be considered as a special case of our notion of infinite-product entropy.
In $\S$\ref{2110020633}, we consider some basic properties of infinite-product entropy, e.g. shift invariance, convexity,
subadditivity with respect to product of probability measures, the behavior with respect to dilation and restriction. We show that for a
translation invariant probability measure the infinite-product entropy coincides with the usual entropy of a shift transformation. We also
consider some basic examples and computations. In $\S$\ref{2110020634}, we prove a variational inequality related to infinite-product entropy
$H^{ip}(\mu)$ of a Borel probability measure $\mu$ and topological entropy of any subset $S\subseteq X^\infty$ with $\r{Support}(\mu)\subset S$
where $X$ is a compact space.

Let $(X,\Sigma)$ be a measurable space i.e. $X$ is a set and $\Sigma$ is a $\sigma$-algebra of subsets of $X$.
By a \emph{partition} $\c{P}$ for $X$ we mean a finite collection $\c{P}=\{A_i\}_{i=1}^k$ of members of $\Sigma$ such that
$\cup_{i=1}^kA_i=X$ and $A_i\cap A_j=\emptyset$ for $i\neq j$. For partitions $\c{P}$ and $\c{Q}=\{B_j\}_j$ we let $\c{P}\vee\c{Q}$
denote the partition $\{A_i\cap B_j\}_{i,j}$. We write $\c{P}\preceq\c{Q}$ when every $A_i$ is a union of $B_j$s. Recall that (\cite{Walters1})
for any probability measure $\nu$ on $X$ and any partition $\c{P}$ of $X$ as above the Shannon entropy associated to the pair $(\nu,\c{P})$ is
defined by $${H^{sh}}(\nu,\c{P}):=-\sum_{i=1}^k(\nu A_i)\log(\nu A_i).$$ (By the convention $0\log 0$ is defined to be $0$.) For any measurable
transformation $T:X\to X$ preserving $\nu$ (i.e. $\nu(T^{-1}A)=\nu(A)$) the entropy of the triple $(T,\nu,\c{P})$ is given by
\begin{equation}\label{2109250537}
{H^{sh}}(T,\nu,\c{P}):=\lim_{n\to\infty}\frac{1}{n}{H^{sh}}(\nu,\vee_{i=0}^{n-1}T^{-i}\c{P}),\end{equation}
and the entropy of $(T,\nu)$ is defined to be the value $${H^{sh}}(T,\nu):=\sup_{\c{P}}{H^{sh}}(T,\nu,\c{P}),$$ where the supremum is taken over
all partitions $\c{P}$ of $X$. Note that the sequence in (\ref{2109250537}) decreases to ${H^{sh}}(T,\nu,\c{P})$ \cite[Theorem 4.10]{Walters1}.

We denote by $X^\infty$ the infinite cartesian product $\prod_{n=0}^\infty X$ and by $\Sigma^\infty$ the product $\sigma$-algebra on $X^\infty$.
For $A_0,\ldots,A_{n-1}$ in $\Sigma$ we let $$\ov{A_0\cdots A_{n-1}}:=A_0\times\cdots\times A_{n-1}\times X\times\cdots\subseteq X^\infty.$$
For any family $\c{A}=\{A_i\}_{i\in I}$ of members of $\Sigma$ and any $n\geq1$ we let
$$\ov{\c{A}}^n:=\Big\{\ov{A_{i_0}\cdots A_{i_{n-1}}}:i_0,\ldots,i_{n-1}\in I\Big\}.$$
Thus $\Sigma^\infty$ is the smallest $\sigma$-algebra on $X^\infty$ containing $\ov{\Sigma}^n$ for every $n\geq1$.
If $\c{P}$ is a partition for $X$ then $\ov{\c{P}}^n$ is a partition for $X^\infty$. Our main definition is as follows:
\begin{definition}\label{2110020750}
Let $\mu$ be a probability measure on the measurable space $(X^\infty,\Sigma^\infty)$. For any partition $\c{P}$ of $X$ the infinite-product entropy
of the pair $(\mu,\c{P})$ is defined by $${H}^{ip}(\mu,\c{P}):=\limsup_{n\to\infty}\frac{1}{n}{H^{sh}}(\mu,\ov{\c{P}}^n).$$
The infinite-product entropy of $\mu$ is defined to be the value $${H}^{ip}(\mu):=\sup_{\c{P}}{H}^{ip}(\mu,\c{P}),$$
where the supremum is taken over all partitions $\c{P}$ of $X$.\end{definition}
In the following it is shown that the entropy of measurable transformations may be considered as the infinite-product entropy of some special class
of probability measures. Let $T:X\to X$ be a measurable transformation and $\nu$ be a probability measure on $X$. Consider the measurable mapping
$\hat{T}:X\to X^\infty$ defined by $\hat{T}(x)=(x,T(x),T^2(x),\ldots)$ and let $\mu_{T,\nu}$ denote the push forward of $\nu$ under $\hat{T}$.
Thus $\mu_{T,\nu}$ is a probability measure on $X^\infty$ given by
\begin{equation}\label{2109250625}
\mu_{T,\nu}\Big(\ov{A_0\cdots A_{n-1}}\Big)=\nu\Big(\cap_{i=0}^{n-1}T^{-i}A_i\Big).\end{equation}
\begin{proposition}
If $T$ is a measure-preserving transformation on $(X,\Sigma,\nu)$ then $${H}^{ip}(\mu_{T,\nu})={H^{sh}}(T,\nu).$$\end{proposition}
\begin{proof}
For any partition $\c{P}$ of $X$ it follows from (\ref{2109250625}) that
$${H^{sh}}(\mu_{T,\nu},\ov{\c{P}}^n)={H^{sh}}(\nu,\vee_{i=0}^{n-1}T^{-i}\c{P}).$$ Thus ${H}^{ip}(\mu_{T,\nu},\c{P})={H^{sh}}(T,\nu,\c{P})$.\end{proof}
\begin{remark}
\emph{As it can be seen above, to define ${H}^{ip}(\mu)$ we only use the finite dimensional marginal distributions of $\mu$.
Thus if for every $n\geq0$, $\mu_{n]}$ is a probability measure on $X^n$ and if the sequence $(\mu_{n]})_n$ is \emph{consistence},
then one can define the entropy value $${H}^{ip}\big((\mu_{n]})_{n\geq0}\big).$$ All of the results given in this note are valid for such
sequences of measures. Of course if $(X,\Sigma)$ is a nice measurable space (e.g. $X$ is a Polish space with its Borel $\sigma$-algebra)
then from Kolmogorov's Extension Theorem we know that $\mu_{n]}$s are marginal distributions of a unique probability
measure $\mu$ on $X^\infty$ and then ${H}^{ip}\big((\mu_{n]})_{n\geq0}\big)=H^{ip}(\mu)$.}\end{remark}
%%%%%%%%%%%%%%%%%%%%%%%%%%%%%%%%%%%%%%%%%%%%%%%%%%%%%%%%%%%%%%%%%%%%%%%%%%%%%%%%%%%%%%%%%%%%%%%%%%%%%%%%%%%%%%%%%%%%%%%%%%
%%%%%%%%%%%%%%%%%%%%%%%%%%%%%%%%%%%%%%%%%%%%%%%%%%%%%%%%%%%%%%%%%%%%%%%%%%%%%%%%%%%%%%%%%%%%%%%%%%%%%%%%%%%%%%%%%%%%%%%%%%
%%%%%%%%%%%%%%%%%%%%%%%%%%%%%%%%%%%%%%%%%%%%%%%%%%%%%%%%%%%%%%%%%%%%%%%%%%%%%%%%%%%%%%%%%%%%%%%%%%%%%%%%%%%%%%%%%%%%%%%%%%
%%%%%%%%%%%%%%%%%%%%%%%%%%%%%%%%%%%%%%%%%%%%%%%%%%%%%%%%%%%%%%%%%%%%%%%%%%%%%%%%%%%%%%%%%%%%%%%%%%%%%%%%%%%%%%%%%%%%%%%%%%
\section{Some basic properties of ${H}^{ip}$}\label{2110020633}
\begin{proposition}\label{2109280630}
Let $\mu$ be a probability measure on $X^\infty$ and let $\c{P},\c{Q}$ be partitions of $X$.
\begin{enumerate}\item[(i)] ${H}^{ip}(\mu,\c{P})\leq\log|\c{P}|$.
\item[(ii)] If $\c{P}\preceq\c{Q}$ then ${H}^{ip}(\mu,\c{P})\leq{H}^{ip}(\mu,\c{Q})$.
\item[(iii)] ${H}^{ip}(\mu,\c{P}\vee\c{Q})\leq{H}^{ip}(\mu,\c{P})+{H}^{ip}(\mu,\c{Q})$.\end{enumerate}\end{proposition}
\begin{proof}By \cite[Corollary 4.2.1]{Walters1} we have ${H^{sh}}(\mu,\ov{\c{P}}^n)\leq n|\c{P}|$. Thus (i) is satisfied.
If $\c{P}\preceq\c{Q}$ then by \cite[Theorem 4.3(iv)]{Walters1} we have ${H^{sh}}(\mu,\ov{\c{P}}^n)\leq{H^{sh}}(\mu,\ov{\c{Q}}^n)$ for every
$n\geq1$ and hence the inequality in (ii) is satisfied. (iii) is proved similarly by \cite[Theorem 4.3(viii)]{Walters1}.\end{proof}
We say that a probability measure $\mu$ on $X^\infty$ is \emph{stationary} if $\mu$ is preserved by the shift mapping
$\f{S}:X^\infty\to X^\infty$ defined by $(x_0,x_1,x_2,\ldots)\mapsto(x_1,x_2,\ldots)$.
\begin{proposition}\label{2109280631}
Let $\mu$ on $X^\infty$ be stationary. Then for any partition $\c{P}$ of $X$,
$${H}^{ip}(\mu,\c{P})={H^{sh}}(\f{S},\mu,\ov{\c{P}}^m)\hspace{10mm}(m\geq1).$$\end{proposition}
\begin{proof} It is easily checked that $\ov{\c{P}}^{m+n-1}=\vee_{i=0}^{n-1}\f{S}^{-i}\ov{\c{P}}^m$. Thus
\begin{equation*}\begin{split}
{H^{sh}}(\f{S},\mu,\ov{\c{P}}^m)&=\lim_{n\to\infty}\frac{1}{n}{H^{sh}}(\mu,\vee_{i=0}^{n-1}\f{S}^{-i}\ov{\c{P}}^m)\\
&=\lim_{n\to\infty}\frac{1}{n}{H^{sh}}(\mu,\ov{\c{P}}^{m+n-1})\\
&=\Big(\lim_{n\to\infty}\frac{m+n-1}{n}\Big)\Big(\limsup_{n\to\infty}\frac{1}{m+n-1}{H^{sh}}(\mu,\ov{\c{P}}^{m+n-1})\Big)\\
&={H}^{ip}(\mu,\c{P}).\end{split}\end{equation*}\end{proof}
\begin{proposition}\label{2109290634}
Let $\mu$ be a stationary probability measure on $X^\infty$. Then $${H}^{ip}(\mu)={H^{sh}}(\f{S},\mu).$$\end{proposition}
\begin{proof} It follows immediately from Proposition \ref{2109280631} that ${H}^{ip}(\mu)\leq{H^{sh}}(\f{S},\mu)$.
Let $\Gamma$ denote the smallest subalgebra of $\Sigma^\infty$ containing $\ov{\Sigma}^n$ for every $n\geq1$. Then every member of
$\Gamma$ is a finite union of sets of the form $\ov{A_0\cdots A_{n-1}}$ ($n\geq1, A_i\in\Sigma$). Since $\Sigma^\infty$
is the smallest $\sigma$-algebra on $X^\infty$ containing $\Gamma$ it follows from \cite[Theorem 4.21]{Walters1} that
$${H^{sh}}(\f{S},\mu)=\sup_{\c{Q}}{H^{sh}}(\f{S},\mu,\c{Q}),$$
where the supremum is taken over all partitions $\c{Q}$ of $X^\infty$ such that $\c{Q}\subset\Gamma$. For every such a partition
$\c{Q}$ it follows from the above observations that there exists a partition $\c{P}$ of $X$ such that $\c{Q}\preceq\ov{\c{P}}^m$ for some
$m\geq1$. By \cite[Theorem 4.12(iii)]{Walters1} we have $${H^{sh}}(\f{S},\mu,\c{Q})\leq{H^{sh}}(\f{S},\mu,\ov{\c{P}}^m).$$
Thus by Proposition \ref{2109280631} we have $${H^{sh}}(\f{S},\mu,\c{Q})\leq{H}^{ip}(\mu,\c{P})\leq{H}^{ip}(\mu).$$
The proof is complete.\end{proof}
Infinite-product entropy is shift invariant:
\begin{proposition}
For probability measure $\mu$ on $X^\infty$ and partition $\c{P}$ of $X$ we have
\begin{equation*}{H}^{ip}(\f{S}_*\mu,\c{P})={H}^{ip}(\mu,\c{P}).\end{equation*}
Thus ${H}^{ip}(\f{S}_*\mu)={H}^{ip}(\mu)$. (Here $\f{S}_*\mu$ denotes the push-forward of $\mu$ under $\f{S}$.)\end{proposition}
\begin{proof} Let $\phi$ denotes the real-valued function $x\mapsto x\log x$ on $[0,\infty)$.
For any finite sequence $r_1,\ldots,r_k$ of nonnegative numbers with $\sum_{i=1}^kr_i=r$ the convexity
of $\phi$ implies that $-\sum_{i=1}^k\phi(r_i)\leq-\phi(r)+r\log k$. Let $\c{P}=\{A_i\}_{i=1}^k$ be a partition for $X$. We have
\begin{equation*}\begin{split}
{H^{sh}}(\mu,\ov{\c{P}}^n)&=\sum_{i_1,\ldots,i_{n-1}=1}^k\Big(-\sum_{i_0=1}^k\phi(\mu(\ov{A_{i_0}A_{i_1}\cdots A_{i_{n-1}}}))\Big)\\
&\leq\sum_{i_1,\ldots,i_{n-1}=1}^k\Big(-\phi(\mu(\ov{XA_{i_1}\cdots A_{i_{n-1}}}))+\mu(\ov{XA_{i_1}\cdots A_{i_{n-1}}})\log k\Big)\\
&={H^{sh}}(\f{S}_*\mu,\ov{\c{P}}^{n-1})+\log k,\end{split}\end{equation*}
and hence ${H}^{ip}(\mu,\c{P})\leq{H}^{ip}(\f{S}_*\mu,\c{P})$. Let $\c{Q}$ be the partition of $X^\infty$ given by
$$\c{Q}:=\Big\{\ov{XA_{i_1}\cdots A_{i_{n-1}}}:1\leq i_1,\ldots,i_{n-1}\leq k\Big\}.$$
We have ${H^{sh}}(\f{S}_*\mu,\ov{\c{P}}^{n-1})={H^{sh}}(\mu,\c{Q})$. Since $\c{Q}\preceq\ov{\c{P}}^{n}$, by
\cite[Theorem 4.3(iv)]{Walters1} we have ${H^{sh}}(\mu,\c{Q})\leq{H^{sh}}(\mu,\ov{\c{P}}^n)$. Thus
${H^{sh}}(\f{S}_*\mu,\ov{\c{P}}^{n-1})\leq{H^{sh}}(\mu,\ov{\c{P}}^n)$, and hence ${H}^{ip}(\f{S}_*\mu,\c{P})\leq{H}^{ip}(\mu,\c{P})$.
The proof is complete.\end{proof}
Infinite-product entropy is a convex function:
\begin{proposition}\label{2109290635}
Let $\mu,\rho$ be probability measures on $X^\infty$. Then
$${H}^{ip}(t\mu+(1-t)\rho)\leq t{H}^{ip}(\mu)+(1-t){H}^{ip}(\rho)\quad\quad(t\in[0,1]).$$\end{proposition}
\begin{proof}It is shown in \cite[proof of Theorem 8.1]{Walters1} that for any partition $\c{P}$ of $X$,
$${H^{sh}}(t\mu+(1-t)\rho,\ov{\c{P}}^n)\leq t{H^{sh}}(\mu,\ov{\c{P}}^n)+(1-t){H^{sh}}(\rho,\ov{\c{P}}^n)+\log2.$$
Thus we have $${H}^{ip}(t\mu+(1-t)\rho,\c{P})\leq t{H}^{ip}(\mu,\c{P})+(1-t){H}^{ip}(\rho,\c{P}).$$
The desired inequality is concluded from the above inequality.\end{proof}
It follows from \cite[Theorem 8.1]{Walters1} and Proposition \ref{2109290634} that the inequality in Proposition \ref{2109290635} is a equality
if $\mu,\rho$ are stationary. Thus infinite-product entropy is an affine function on the convex set of
stationary probability measures on $X^\infty$.

Let $(r_n)_{n\geq0}$ be a strictly increasing sequence of nonnegative integers. The associated \emph{restriction} $\f{R}$
on $X^\infty$ is the mapping defined by $$\f{R}:X^\infty\to X^\infty\hspace{10mm}(x_0,x_1,\ldots)\mapsto(x_{r_0},x_{r_1},\ldots).$$
\begin{proposition}\label{2109300610}
With the above notations suppose that there exists $k\geq1$ such that $r_n\leq(n+1)k-1$ for every
$n\geq0$. Then for any probability measure $\mu$ on $X^\infty$ we have $${H}^{ip}(\f{R}_*\mu)\leq k{H}^{ip}(\mu).$$\end{proposition}
\begin{proof}For any partition $\c{P}$ of $X$ since $\f{R}^{-1}(\ov{\c{P}}^n)\preceq\ov{\c{P}}^{kn}$ ($n\geq1$) we have
$${H^{sh}}(\f{R}_*\mu,\ov{\c{P}}^n)={H^{sh}}(\mu,\f{R}^{-1}(\ov{\c{P}}^n))\leq{H^{sh}}(\mu,\ov{\c{P}}^{kn}),$$
and hence ${H}^{ip}(\f{R}_*\mu,\c{P})\leq k{H}^{ip}(\mu,\c{P})$. The proof is complete.\end{proof}
For $k\geq1$ the \emph{$k$-dilation} on $X^\infty$ is the mapping $$\f{D}^{(k)}:X^\infty\to X^\infty\hspace{10mm}
(x_0,x_1,\ldots)\mapsto(\overbrace{x_0,\ldots,x_0}^k,\overbrace{x_1,\ldots,x_1}^k,\ldots).$$
\begin{proposition}
For any probability measure $\mu$ on $X^\infty$ we have
$${H}^{ip}(\f{D}^{(k)}_*\mu)=\frac{1}{k}{H}^{ip}(\mu).$$\end{proposition}
\begin{proof} Let $\c{P}=\{A_i\}$ be a partition for $X$. Let $n\geq1$ and suppose that $[\frac{n}{k}]$ denotes the smallest integer greater than or
equal to $\frac{n}{k}$. For any member $B=\ov{A_{i_0}\cdots A_{i_{n-1}}}$ of $\c{P}^n$ if
\begin{equation}\label{2109300611}\big(A_{i_0}=\cdots=A_{i_{k-1}}\big)\hspace{3mm}\big(A_{i_k}=\cdots=A_{i_{2k-1}}\big)\hspace{3mm}
\cdots\hspace{3mm}\big(A_{i_{([\frac{n}{k}]-1)k}}=\cdots=A_{i_{n-1}}\big)\end{equation}
then $(\f{D}^{(k)}_*\mu)(B)=\mu(\ov{B_{0}\cdots B_{[\frac{n}{k}]-1}})$ where $B_{\ell}=A_{i_{\ell k}}$; and if (\ref{2109300611}) does not hold
then $(\f{D}^{(k)}_*\mu)(B)=0$. Thus we have ${H^{sh}}(\f{D}^{(k)}_*\mu,\ov{\c{P}}^n)={H^{sh}}(\mu,\ov{\c{P}}^{[\frac{n}{k}]})$. This implies that
$${H}^{ip}(\f{D}^{(k)}_*\mu,\c{P})=\frac{1}{k}{H}^{ip}(\mu,\c{P}).$$ The proof is complete.\end{proof}
Let $f:X\to Y$ be a measurable map between measurable spaces. We denote by $f^\infty$ the mapping $X^\infty\to Y^\infty$ defined by
$(x_1,x_2,\ldots)\mapsto(fx_1,fx_2,\ldots)$.
\begin{proposition}\label{2109300640}
For every probability measure $\mu$ on $X^\infty$ we have
$${H}^{ip}(f^\infty_*\mu)\leq{H}^{ip}(\mu).$$\end{proposition}
\begin{proof}It follows from the easily verified equality ${H}^{ip}(f^\infty_*\mu,\c{Q})={H}^{ip}(\mu,f^{-1}\c{Q})$ for every partition
$\c{Q}$ of $Y$. (Indeed we have ${H^{sh}}(f^\infty_*\mu,\ov{\c{Q}}^n)={H^{sh}}(\mu,\ov{f^{-1}\c{Q}}^n)$ for every $n\geq1$.)\end{proof}
\begin{proposition}\label{2109300612}
Let $\pi$ be a probability measure on $(X\times Y)^\infty$. Let $\mu,\rho$ denote the marginal
measures of $\pi$ respectively on $X^\infty$ and $Y^\infty$: $$(\r{proj}_{X^\infty})_*\pi=\mu\quad\text{and}\quad(\r{proj}_{Y^\infty})_*\pi=\rho.$$
The following inequality holds: $$\r{max}\{{H}^{ip}(\mu),{H}^{ip}(\rho)\}\leq{H}^{ip}(\pi)\leq{H}^{ip}(\mu)+{H}^{ip}(\rho).$$\end{proposition}
\begin{proof}Similar to the proof of Proposition \ref{2109290634} it follows from \cite[Theorem 4.21]{Walters1} that
$${H}^{ip}(\pi)=\sup{H}^{ip}(\pi,\c{P}\odot\c{Q}),$$ where the supremum is taken over all partitions $\c{P}\odot\c{Q}:=\{A_i\times B_j\}_{i,j}$
such that $\c{P}=\{A_i\}_i$ and $\c{Q}=\{B_j\}_j$ are respectively partitions of $X$ and $Y$. We have
$${H^{sh}}(\mu,\ov{\c{P}}^n)={H^{sh}}(\pi,\r{proj}_{X^\infty}^{-1}\ov{\c{P}}^n)
\hspace{10mm}{H^{sh}}(\rho,\ov{\c{Q}}^n)={H^{sh}}(\pi,\r{proj}_{Y^\infty}^{-1}\ov{\c{Q}}^n)$$
$$\ov{\c{P}\odot\c{Q}}^n=(\r{proj}_{X^\infty}^{-1}\ov{\c{P}}^n)\vee(\r{proj}_{X^\infty}^{-1}\ov{\c{P}}^n)$$
Thus ${H^{sh}}(\pi,\ov{\c{P}\odot\c{Q}}^n)\leq{H^{sh}}(\mu,\ov{\c{P}}^n)+{H^{sh}}(\rho,\ov{\c{Q}}^n)$. This implies that
$${H}^{ip}(\pi,\c{P}\odot\c{Q})\leq{H}^{ip}(\mu,\c{P})+{H}^{ip}(\rho,\c{Q}).$$ Hence ${H}^{ip}(\pi)\leq{H}^{ip}(\mu)+{H}^{ip}(\rho)$.
The other inequality follows from Proposition \ref{2109300640}.\end{proof}
\begin{proposition}
Let $\mu,\rho$ be probability measures respectively on $X^\infty,Y^\infty$. Suppose that at least one of these measures is stationary. Then
$${H}^{ip}(\mu\times\rho)={H}^{ip}(\mu)+{H}^{ip}(\rho),$$ where $\mu\times\rho$ is considered as a measure on
$(X\times Y)^\infty\cong(X^\infty\times Y^\infty)$.\end{proposition}
\begin{proof} First of all note that if $\sum_ir_i=1$ and $\sum_js_j=1$ where $r_i,s_j\geq0$ then
\begin{equation}\label{2110010700}\sum_{i,j}\phi(r_is_j)=\sum_i\phi(r_i)+\sum_j\phi(s_j).\end{equation}
We suppose that $\rho$ is stationary and ${H}^{ip}(\mu),
{H}^{ip}(\rho)<\infty$. Let $\epsilon>0$ be arbitrary and fixed. There exist a partition $\c{P}$ of $X$ and a subsequence $(n_k)_k$
of natural numbers such that the following inequality holds.
$${H}^{ip}(\mu)-\frac{\epsilon}{2}\leq\lim_{k\to\infty}\frac{1}{n_k}{H^{sh}}(\mu,\ov{\c{P}}^{n_k})
\leq\limsup_{n\to\infty}\frac{1}{n}{H^{sh}}(\mu,\ov{\c{P}}^{n}).$$ Also there exist a partition $\c{Q}$ of $Y$ such that
$${H}^{ip}(\rho)-\frac{\epsilon}{2}\leq\lim_{n\to\infty}\frac{1}{n}{H^{sh}}(\rho,\ov{\c{Q}}^{n}).$$ We have
\begin{equation*}\begin{split}{H}^{ip}(\mu)+{H}^{ip}(\rho)-\epsilon&\leq\lim_{k\to\infty}\frac{1}{n_k}
\big[{H^{sh}}(\mu,\ov{\c{P}}^{n_k})+{H^{sh}}(\rho,\ov{\c{Q}}^{n_k})\big]\\
&=\lim_{k\to\infty}\frac{1}{n_k}{H^{sh}}(\mu\times\rho,\ov{\c{P}\odot\c{Q}}^{n_k})\\
&\leq\limsup_{n\to\infty}\frac{1}{n}{H^{sh}}(\mu\times\rho,\ov{\c{P}\odot\c{Q}}^{n})\\
&={H}^{ip}(\mu\times\rho,\c{P}\odot\c{Q})\\
&\leq{H}^{ip}(\mu\times\rho).\end{split}\end{equation*}Thus ${H}^{ip}(\mu\times\rho)\geq{H}^{ip}(\mu)+{H}^{ip}(\rho)$.
The reverse inequality follows from Proposition \ref{2109300612}.\end{proof}
\begin{proposition}\label{2110040655}
Let $\mu$ be a probability measure on $X^\infty$ and $k\geq1$. Let ${H}^{ip}_{X^k}(\mu)$ denote the entropy of
$\mu$ as a measure on the space $(X^k)^\infty$ identified with $X^\infty$ under the mapping $$(X^k)^\infty\to X^\infty\quad
\big((x_0^1,\ldots,x^k_0),(x_1^1,\ldots,x_1^k),\ldots\big)\mapsto\big(x_0^1,\ldots,x^k_0,x_1^1,\ldots,x_1^k,\ldots\big).$$
Then we have $${H}^{ip}_{X^k}(\mu)=k{H}^{ip}_{X}(\mu).$$\end{proposition}
\begin{proof} First of all note that $${H}^{ip}_{X^k}(\mu)=\sup{H}^{ip}_{X^k}(\mu,\c{P}_1\odot\cdots\odot\c{P}_k)$$
where the supremum is taken over all $k$-tuples $(\c{P}_1,\ldots,\c{P}_k)$ of partitions of $X$. We have
\begin{equation*}\begin{split}{H}^{ip}_{X^k}(\mu,\odot_{i=1}^k\c{P}_i)&=\limsup_{n\to\infty}\frac{1}{n}
{H^{sh}}(\mu,\ov{\odot_{i=1}^k\c{P}_i}^{n})\\
&\leq\limsup_{n\to\infty}\frac{1}{n}{H^{sh}}(\mu,\ov{\vee_{i=1}^k\c{P}_i}^{nk})\\
&\leq k{H}^{ip}_{X}(\mu,\vee_{i=1}^k\c{P}_i).\end{split}\end{equation*}
This shows that ${H}^{ip}_{X^k}(\mu)\leq k{H}^{ip}_{X}(\mu)$. For any partition $\c{P}$ of $X$ we have
\begin{equation*}\begin{split}{H}^{ip}_{X}(\mu,\c{P})&=\limsup_{n\to\infty}\frac{1}{n}{H^{sh}}(\mu,\ov{\c{P}}^n)\\
&\leq\limsup_{n\to\infty}\frac{1}{n}{H^{sh}}(\mu,\ov{\odot_{i=1}^k\c{P}}^{[\frac{n}{k}]})\\
&\leq \frac{1}{k}{H}^{ip}_{X^k}(\mu,\odot_{i=1}^k\c{P}).\end{split}\end{equation*}
Thus ${H}^{ip}_{X^k}(\mu)\geq k{H}^{ip}_{X}(\mu)$. The proof is complete.\end{proof}
We end this section by some examples and simple computations:
\begin{example}
\emph{Let $(\nu_n)_{n\geq0}$ be a sequence of probability measures on a measurable space $X$. It is well-known that the product probability measure
$\prod_{n=0}^\infty\nu_n$ exists on $X^\infty$ \cite[$\S$38]{Halmos1}. For any measurable partition $\c{P}$ of $X$ by (\ref{2110010700}) we have
$${H}^{ip}(\prod_{n=0}^\infty\nu_n,\c{P})=\limsup_{n\to\infty}\frac{1}{n}\sum_{i=0}^{n-1}{H^{sh}}(\nu_i,\c{P}).$$  In case that
$\c{P}=\{A_1,\ldots,A_k\}$ and there is a measurable mapping $T:X\to X$ and a probability measure $\mu$ on $X$ such that $\nu_n:=(T^n)_*\nu$, we have
$${H}^{ip}(\prod_{n=0}^\infty\nu_n,\c{P})=-\liminf_{n\to\infty}\frac{1}{n}\sum_{i=0,\ldots,n-1}^{j=1,\ldots,k}\nu(T^{-i}A_k)\log\nu(T^{-i}A_k).$$
In case that $\nu_n=\nu$ for every $n$, we have $${H}^{ip}(\prod_{n=0}^\infty\nu,\c{P})={H^{sh}}(\nu,\c{P}).$$
In particular, if $X:=\{1,\ldots,k\}$ then we have $${H}^{ip}(\prod_{n=0}^\infty\nu)=-\sum_{i=1}^k\nu(i)\log\nu(i).$$}\end{example}
\begin{example}
\emph{Let $X=\{1,\ldots,k\}$. If $\c{P}$ denotes the partition $\{\{1\},\ldots,\{k\}\}$ then for any probability measure $\mu$ on $X^\infty$
by Proposition \ref{2109280630} we have $H^{ip}(\mu)=H^{ip}(\mu,\c{P})\leq\log k$. Let $\mu$ be the distribution of a homogenous $X$-valued
Markov chain with stochastic matrix $(p_{ij})_{i,j=1}^k$ and invariant initial probability vector $(p_1,\ldots,p_k)$. By
\cite[Theorem 4.27]{Walters1} and Proposition \ref{2109290634} we have $${H}^{ip}(\mu)=-\sum_{i,j=1}^kp_ip_{ij}\log p_{ij}.$$}
\end{example}
\begin{example}
\emph{Let $k\ge2$ be a natural number. For any $r\in[0,\infty]$ let $\f{B}^{(k)}(r)$ denote the representation of $r$ in base $k$.
For those rational numbers $r$ that have two $k$-representations we choose the lower representation. Then $\f{B}^{(k)}:[0,1]\to X^\infty$ is a
well-defined injective measurable mapping where $X:=\{0,\ldots,k-1\}$. For any probability Borel measure $\nu$ on $[0,1]$ it is natural to
consider the value $H^{ip}(\f{B}^{(k)}_*\nu)$ as the \emph{$k$-base entropy number} of $\nu$. It is easily seen that for the Lebesgue
measure this number is equal to the maximum value $\log k$. For any point-mass measure this number is $0$.}
\end{example}
%%%%%%%%%%%%%%%%%%%%%%%%%%%%%%%%%%%%%%%%%%%%%%%%%%%%%%%%%%%%%%%%%%%%%%%%%%%%%%%%%%%%%%%%%%%%%%%%%%%%%%%%%%%%%%%%%%%%%%%%%%
%%%%%%%%%%%%%%%%%%%%%%%%%%%%%%%%%%%%%%%%%%%%%%%%%%%%%%%%%%%%%%%%%%%%%%%%%%%%%%%%%%%%%%%%%%%%%%%%%%%%%%%%%%%%%%%%%%%%%%%%%%
%%%%%%%%%%%%%%%%%%%%%%%%%%%%%%%%%%%%%%%%%%%%%%%%%%%%%%%%%%%%%%%%%%%%%%%%%%%%%%%%%%%%%%%%%%%%%%%%%%%%%%%%%%%%%%%%%%%%%%%%%%
%%%%%%%%%%%%%%%%%%%%%%%%%%%%%%%%%%%%%%%%%%%%%%%%%%%%%%%%%%%%%%%%%%%%%%%%%%%%%%%%%%%%%%%%%%%%%%%%%%%%%%%%%%%%%%%%%%%%%%%%%%
\section{A Variational Inequality}\label{2110020634}
Let $X$ be a compact topological space. In \cite{Sadr1} we have defined an entropy number for any arbitrary subset $S$ of $X^\infty$.
In this note we call that number \emph{topological infinite-product entropy} of $S$ and denote it by $H^{tip}(S)$.
Let us quickly review the definition: For any open cover $\c{U}$ of $X$ the family $\ov{\c{U}}^n$ ($n\geq1$) is an open cover for $X^\infty$.
Let $$N_S(\ov{\c{U}}^n):=\min|\c{A}|\hspace{5mm}\text{ant}\hspace{5mm}
H^{tip}(S,\c{U}):=\limsup_{n\to\infty}\frac{1}{n}N_S(\ov{\c{U}}^n)$$
where the minimum is taken over all subfamilies $\c{A}\subseteq\ov{\c{U}}^n$ with $S\subset\cup\c{A}$.
Then $H^{tip}(S)$ is defined to be the value $\sup H^{tip}(S,\c{U})$ where the supremum is taken over all open covers $\c{U}$ of $X$.
In the following we generalize a one half of \emph{Variational Principle} (\cite[$\S$8.2]{Walters1}) for $S$.
\begin{lemma}\label{2110040650}
Let $\c{P},\c{Q}$ be partitions of $X$ and $\mu$ stationary Borel measure on $X^\infty$. Then
$${H}^{ip}(\mu,\c{P})\leq{H}^{ip}(\mu,\c{Q})+H^{sh}(\mu,\ov{\c{P}}^1/\ov{\c{Q}}^1),$$
where $H^{sh}(\mu,\ov{\c{P}}^1/\ov{\c{Q}}^1)$ denotes the conditional entropy of $\ov{\c{P}}^1$ given (the $\sigma$-algebra generated by)
$\ov{\c{Q}}^1$ with respect to $\mu$ \cite[$\S$4.3]{Walters1}.\end{lemma}
\begin{proof}
It follows from \cite[Theorem 4.12(iv)]{Walters1} and identities
$${H}^{ip}(\mu,\c{P})={H^{sh}}(\f{S},\mu,\ov{\c{P}}^1)\quad\text{and}\quad{H}^{ip}(\mu,\c{Q})={H^{sh}}(\f{S},\mu,\ov{\c{Q}}^1).$$\end{proof}
\begin{proposition}\label{2110040657}
Variational Inequality: Let $X$ be a compact metrizable space endowed with its Borel $\sigma$-algebra.
Let $S\subseteq X^\infty$ be an arbitrary subset. Then for any stationary probability measure $\mu$ on $X^\infty$ with
$\r{Support}(\mu)\subseteq S$ we have $${H}^{ip}(\mu)\leq H^{tip}(S).$$\end{proposition}
\begin{proof}
Let $\nu$ denotes the marginal of $\mu$ on the first component of $X^\infty$ i.e. the push-forward of $\mu$ under the mapping
$(x_0,x_1,\ldots)\mapsto x_0$. Since $X$ is compact and metrizable, $\nu$ is regular. Let $\c{P}=\{A_i\}_{i=1}^k$ be a (Borel) partition
for $X$. Then as it is denoted in \cite[Proof of Theorem 8.2]{Walters1} there exist compact sets $B_i\subseteq A_i$ ($i=1,\ldots,k$)
such that $$H^{sh}(\mu,\ov{\c{P}}^1/\ov{\c{Q}}^1)=H^{sh}(\nu,\c{P}/\c{Q})\leq1,$$ where $\c{Q}$ denotes the partition $\{B_0,B_1,\ldots,B_k\}$
of $X$ with $B_0:=X\setminus(\cup_{i=1}^kB_k)$. Let $\c{U}$ denote the open cover $\{U_i\}_{i=1}^k$ of $X$ where $U_i:=B_0\cup B_i$.
Then we have $${H^{sh}}(\mu,\ov{\c{Q}}^n)\leq\log\big(2^n{N}_{\r{Support}(\mu)}(\ov{\c{U}}^n)\big)\leq\log\big(2^n{N}_{S}(\ov{\c{U}}^n)\big).$$
Thus $${H}^{ip}(\mu,\c{Q})\leq{H}^{tip}(S,\c{U})+\log2\leq{H}^{tip}(S)+\log2,$$ and hence by Lemma \ref{2110040650} we have
$${H}^{ip}(\mu,\c{P})\leq{H}^{tip}(S)+\log2+1.$$ Since $\c{P}$ is an arbitrary partition, the above inequality implies that
\begin{equation}\label{2110040651}{H}^{ip}(\mu)\leq{H}^{tip}(S)+\log2+1.\end{equation}
The inequality (\ref{2110040651}) holds for any triple $(X,\mu,S)$ satisfying the assumptions of the proposition. Thus if we imagine $\mu$
as a measure on $(X^n)^\infty$ (in the same way that we did in Proposition \ref{2110040655}) and $S$ as a subset of $(X^n)^\infty$ then by
Proposition \ref{2110040655} and \cite[Theorem 3.15]{Sadr1} the inequality (\ref{2110040651}) becomes
$$n{H}^{ip}(\mu)\leq n{H}^{tip}(S)+\log2+1.$$ The proof is complete.\end{proof}
In the following we give another proof of Proposition \ref{2110040657}:
We know that $\r{Support}(\mu)\subseteq X^\infty$ is compact. Since $\mu$ is stationary it can be easily checked that $\r{Support}(\mu)$ is
$\f{S}$-invariant. Thus $\f{S}$ may be regarded as a measure-preserving transformation on the probability space $(\r{Support}(\mu),\mu)$
and then we have $${H}^{ip}(\mu)={H^{sh}}(\f{S},\mu)={H^{sh}}(\f{S}|_{\r{Support}(\mu)},\mu).$$ Thus it follows from the Variational Principle
\cite[Theorem 8.2]{Walters1} that $${H}^{ip}(\mu)\leq h(\f{S}|_{\r{Support}(\mu)}),$$ where $h$ denotes topological entropy of continuous
endomorphisms. It can be shown that the value $h(\f{S}|_{\r{Support}(\mu)})$ is equal to ${H}^{tip}(\r{Support}(\mu))$. Now since
$\r{Support}(\mu)\subseteq S$ from \cite[Theorem 3.3]{Sadr1} we have ${H}^{tip}(\r{Support}(\mu))\leq{H}^{tip}(S)$.
Thus ${H}^{ip}(\mu)\leq{H}^{tip}(S)$.
%%%%%%%%%%%%%%%%%%%%%%%%%%%%%%%%%%%%%%%%%%%%%%%%%%%%%%%%%%%%%%%%%%%%%%%%%%%%%%%%%%%%%%%%%%%%%%%%%%%%%%%%%%%%%%%%%%%%%%%%%%
%%%%%%%%%%%%%%%%%%%%%%%%%%%%%%%%%%%%%%%%%%%%%%%%%%%%%%%%%%%%%%%%%%%%%%%%%%%%%%%%%%%%%%%%%%%%%%%%%%%%%%%%%%%%%%%%%%%%%%%%%%
%%%%%%%%%%%%%%%%%%%%%%%%%%%%%%%%%%%%%%%%%%%%%%%%%%%%%%%%%%%%%%%%%%%%%%%%%%%%%%%%%%%%%%%%%%%%%%%%%%%%%%%%%%%%%%%%%%%%%%%%%%
%%%%%%%%%%%%%%%%%%%%%%%%%%%%%%%%%%%%%%%%%%%%%%%%%%%%%%%%%%%%%%%%%%%%%%%%%%%%%%%%%%%%%%%%%%%%%%%%%%%%%%%%%%%%%%%%%%%%%%%%%%
%%%%%%%%%%%%%%%%%%%%%%%%%%%%%%%%%%%%%%%%%%%%%%%%%%%%%%%%%%%%%%%%%%%%%%%%%%%%%%%%%%%%%%%%%%%%%%%%%%%%%%%%%%%%%%%%%%%%%%%%%%

%%%%%%%%%%%%%%%%%%%%%%%%%%%%%%%%%%%%%%%%%%%%%%%%%%%%%%%%%%%%%%%%%%%%%%%%%%%%%%%%%%%%%%%%%%%%%%%%%%%%%%%%%%%%%%%%%%%%%%%%%%
%%%%%%%%%%%%%%%%%%%%%%%%%%%%%%%%%%%%%%%%%%%%%%%%%%%%%%%%%%%%%%%%%%%%%%%%%%%%%%%%%%%%%%%%%%%%%%%%%%%%%%%%%%%%%%%%%%%%%%%%%%
%%%%%%%%%%%%%%%%%%%%%%%%%%%%%%%%%%%%%%%%%%%%%%%%%%%%%%%%%%%%%%%%%%%%%%%%%%%%%%%%%%%%%%%%%%%%%%%%%%%%%%%%%%%%%%%%%%%%%%%%%%
%%%%%%%%%%%%%%%%%%%%%%%%%%%%%%%%%%%%%%%%%%%%%%%%%%%%%%%%%%%%%%%%%%%%%%%%%%%%%%%%%%%%%%%%%%%%%%%%%%%%%%%%%%%%%%%%%%%%%%%%%%
%%%%%%%%%%%%%%%%%%%%%%%%%%%%%%%%%%%%%%%%%%%%%%%%%%%%%%%%%%%%%%%%%%%%%%%%%%%%%%%%%%%%%%%%%%%%%%%%%%%%%%%%%%%%%%%%%%%%%%%%%%
\end{document}